\documentclass[12pt,a4paper]{amsart}

\usepackage{amsmath,amsthm,amssymb}
\usepackage[all]{xy}
\usepackage{latexsym}

\usepackage[top=30truemm,bottom=30truemm,left=25truemm,right=25truemm]{geometry}

\theoremstyle{plain}
\newtheorem{definition}{Definition}[section]
\newtheorem{thm}[definition]{Theorem}

\newtheorem{lem}[definition]{Lemma}
\newtheorem{cor}[definition]{Corollary}
\newtheorem{rei}[definition]{Example}

\newcommand{\sgn}{{\rm sgn}}
\newcommand{\Mat}{{\rm M}}
\newcommand{\diag}{{\rm diag}}
\newcommand{\Sdet}{{\rm Sdet}}
\newcommand{\Mdet}{{\rm Mdet}}
\newcommand{\Ddet}{{\rm Ddet}}
\newcommand{\Det}{{\rm Det}}
\newcommand{\Hom}{{\rm Hom}}

\begin{document}
\title[Study-type determinants and their properties]{Study-type determinants and their properties}
\author[N. Yamaguchi]{Naoya Yamaguchi}
\date{\today}

\maketitle

\begin{abstract}
In this paper, we define the concept of the Study-type determinant, 
and we present some properties of these determinants. 
These properties lead to some properties of the Study determinant. 
The properties of the Study-type determinants are obtained using a commutative diagram. 
This diagram leads not only to these properties, 
but also to an inequality for the degrees of representations and to an extension of Dedekind's theorem. 
\end{abstract}

\section{Introduction}
Let $B$ be a commutative ring, 
let $A$ be a ring\footnote{Here rings are assumed to possess a multiplicative unit.} that is a free right $B$-module of rank $m$, 
and let $L$ be a left regular representation from $\Mat(r, A)$ to $\Mat(m, \Mat(r, B))$, 
where $\Mat(r, A)$ is the set of all $r \times r$ matrices with elements in $A$. 
We define the Study-type determinant $\Sdet_{\Mat(r, A)}^{B} : \Mat(r, A) \rightarrow B$ as 
$$
\Sdet_{\Mat(r, A)}^{B} := \Det_{\Mat(mr, B)}^{B} \circ \iota \circ L, 
$$
where $\iota$ is the inclusion map from $\Mat(m, \Mat(r, B))$ to $\Mat(mr, B)$. 
The Study-type determinant has the following properties. 

For all $a, a' \in \Mat(r, A)$, the following hold: 
\begin{enumerate}
\item[(S1)] $\Sdet_{\Mat(r, A)}^{B}(a a') = \Sdet_{\Mat(r, A)}^{B}(a) \: \Sdet_{\Mat(r, A)}^{B}(a')$; 
\item[(S2)] $a$ is invertible in $\Mat(r, A)$ if and only if $\Sdet_{\Mat(r, A)}^{B}(a)$ is invertible in $B$; 
\item[(S3)] if $a'$ is obtained from $a$ by adding a left-multiple of a row to another row or a right-multiple of a column to another column, 
then we have $\Sdet_{\Mat(r, A)}^{B}(a') = \Sdet_{\Mat(r, A)}^{B}(a)$; 
\item[(S4)] if there exists a basis $( e_{1} \: e_{2} \: \cdots \: e_{m} )$ of $A$ as a $B$-right module satisfying the following conditions: 
\begin{enumerate}
\item[(i)] $e_{i}$ is invertible in $A$ for any $e_{i}$;  
\item[(ii)] $e_{i}^{-1} B e_{i} \subset B$ holds for any $e_{i}$, 
\end{enumerate}
then we have $\Sdet_{\Mat(r, A)}^{B}(a) \in Z(A) \cap B$, where $Z(A)$ is the center of $A$. 
\end{enumerate}
In the following, we assume that $B$ is a free right $C$-module of rank $n$ and $C$ is a commutative ring. 
Then we have the following properties.\footnote{The set $e_{i}B$ and the product $*$ are explained in Section~$7$.}

For all $a \in \Mat(r, A)$, the following hold: 
\begin{enumerate}
\setcounter{enumi}{4}
\item[(S5)] $\Sdet_{\Mat(r, A)}^{C} = \Sdet_{\Mat(1, B)}^{C} \circ \Sdet_{\Mat(r, A)}^{B}$, where we regard $B$ as $\Mat(1, B)$; 
\item[(S6)] if there exists a basis of $A$ as a $B$-right module satisfying the conditions (i) and (ii), 
then we have $\Sdet_{\Mat(r, A)}^{C}(a) = \left( \Sdet_{\Mat(r, A)}^{B}(a) \right)^{n}$. 
\end{enumerate}

The Study-type determinant is a generalization of the Study determinant. 
The Study determinant was defined by Eduard Study~\cite{study1920}. 
Let $\mathbb{H}$ be the quaternion field. 
The Study determinant $\Sdet : \Mat(r, \mathbb{H}) \rightarrow \mathbb{C}$ is defined using a transformation from $\psi_{r} \colon \Mat(r, \mathbb{H}) \to \Mat(2r, \mathbb{C})$. 
It is known that this determinant has the following properties\footnote{The algebra homomorphisms $\phi_{2r}$ and $\psi_{r}$ are explained in Section~$9$.} (see, e.g., \cite{Aslaksen1996}). 

For all $a, a' \in \Mat(r, \mathbb{H})$, the following hold: 
\begin{enumerate}
\item[(S$1'$)] $\Sdet(a a') = \Sdet(a) \: \Sdet(a')$; 
\item[(S$2'$)] $a$ is invertible in $\Mat(r, \mathbb{H})$ if and only if $\Sdet(a) \in \mathbb{C}$ is invertible; 
\item[(S$3'$)] if $a'$ is obtained from $a$ by adding a left-multiple of a row to another row or a right-multiple of a column to another column, 
then we have $\Sdet(a') = \Sdet(a)$; 
\item[(S$4'$)] $\Sdet(a) \in Z(\mathbb{H}) = \mathbb{R}$; 
\item[(S$5'$)] $\left( \Det_{\Mat(4r, \mathbb{R})}^{\mathbb{R}} \circ \phi_{2r} \circ \psi_{r} \right)(a) = \Sdet{(a)}^{2}$. 
\end{enumerate}
The above properties can be derived from the properties of the Study-type determinants. 

Let $L_{2}$ and $L_{3}$ be left regular representations from $B$ to $\Mat(n, C)$ and from $\Mat(r, B)$ to $\Mat(n, \Mat(r, C))$, respectively. 
The following theorem plays an important role in ascertaining the properties of the Study-type determinants. 
\begin{thm}[see Theorem~$\ref{thm:4.4}$ for proof]\label{thm:1.1}
The following diagram is commutative: 
\[
\xymatrix{
\Mat(r, B) \ar[d]^-{L_{3}} \ar[r]^-{\Det_{\Mat(r, B)}^{B}} & B \ar[r]^-{L_{2}} & \Mat(n, C) \ar[d]^-{\Det_{\Mat(n, C)}^{C}} \\ 
\Mat(n, \Mat(r, C)) \ar@{^{(}-_>}[r]^-{} & \Mat(nr, C) \ar[r]^-{\Det_{\Mat(nr, C)}^{C}} & C \\ 
}
\]
\end{thm}

Theorem~$\ref{thm:1.1}$ leads not only to some properties of the Study-type determinant, 
but also to Corollary~$\ref{cor:1.2}$ and Theorem~$\ref{thm:1.3}$. 
Let $e = (e_{1} \: e_{2} \: \cdots \: e_{m})$ be a basis of $A$ as $B$-module, 
let $f = (f_{1} \: f_{2} \: \cdots \: f_{m})$ be a basis of $B$ as $C$-module, 
let $ef := \left\{ e_{i} f_{j} \mid i \in [m], j \in [n] \right\}$, 
let $\left\{ x_{\alpha} \mid \alpha \in ef \right\}$ be the set of independent commuting variables, 
and let $\mathfrak{X}_{ef} = \mathfrak{X} := \sum_{\alpha \in ef} \alpha x_{\alpha} \in A[x_{\alpha}]$ be the general element for $ef$, 
where $A[x_{\alpha}]$ is the polynomial ring in $\left\{ x_{\alpha} \mid \alpha \in ef \right\}$ with coefficients in $A$. 
For rings $R$ and $R'$, 
we denote the set of ring homomorphisms from $R$ to $R'$ by $\Hom(R, R')$, 
and we regard any ring homomorphism $\rho \in \Hom(R, R')$ as $\rho \in \Hom(R[x_{\alpha}], R'[x_{\alpha}])$ such that $\rho(x_{\alpha} 1_{R}) = x_{\alpha} \rho(1_{R})$ for any $\alpha \in ef$, 
where $1_{R}$ is the unit element of $R$. 
Let $\rho \in \Hom(A, \Mat(r, B))$. 
We assume that there exists a commutative ring such that $C \subset \bar{C}$, 
and $L_{3} \circ \rho$ and $L_{2}$ have the direct sums 
\begin{align*}
L_{3} \circ \rho &\sim \varphi^{(1)} \oplus \varphi^{(2)} \oplus \cdots \oplus \varphi^{(s)}, &&\varphi^{(i)}(a) \in \Mat(r_{i}, \bar{C}), \\ 
L_{2} &\sim \psi^{(1)} \oplus \psi^{(2)} \oplus \cdots \oplus \psi^{(t)}, &&\psi^{(j)}(b) \in \Mat(n_{j}, \bar{C}), 
\end{align*}
where $a \in A$ and $b \in B$. 
Then, we have the following corollary and theorem. 

\begin{cor}[Corollary~$\ref{cor:5.1}$]\label{cor:1.2}
The following holds: 
\begin{align*}
\prod_{1 \leq i \leq s} \left( \Det_{\Mat(r_{i}, \bar{C}[x_{\alpha}])}^{\bar{C}[x_{\alpha}]} \circ \varphi^{(i)} \right) (\mathfrak{X}) 
&= \left( \Det_{\Mat(nr, C[x_{\alpha}])}^{C[x_{\alpha}]} \circ L_{3} \circ \rho \right) (\mathfrak{X}) \\ 
&= \left( \Det_{\Mat(n, C[x_{\alpha}])}^{C[x_{\alpha}]} \circ L_{2} \circ \Det_{\Mat(r, B[x_{\alpha}])}^{B[x_{\alpha}]} \circ \rho \right) (\mathfrak{X}) \\
&= \prod_{1 \leq j \leq t} \left( \Det_{\Mat(n_{j}, \bar{C}[x_{\alpha}])}^{\bar{C}[x_{\alpha}]} \circ \psi^{(j)} \circ \Det_{\Mat(r, B[x_{\alpha}])}^{B[x_{\alpha}]} \circ \rho \right) (\mathfrak{X}). 
\end{align*}
\end{cor}

\begin{thm}[see Theorem~$\ref{thm:5.3}$ for proof]\label{thm:1.3}
If $\left( \Det_{\Mat(r_{i}, \bar{C}[x_{\alpha}])}^{\bar{C}[x_{\alpha}]} \circ \varphi^{(i)} \right) (\mathfrak{X})$ is an irreducible polynomial over $\bar{C}[x_{\alpha}]$, 
then we have 
$$
\deg{\varphi^{(i)}} \leq \max{ \left\{ \deg{\psi^{(j)}} \mid 1 \leq j \leq t \right\} } \times \deg{\rho}, 
$$
where $\deg{\rho}$ is the degree of the polynomial $\left( \Det_{\Mat(r, B[x_{\alpha}])}^{B[x_{\alpha}]} \circ \rho \right)(\mathfrak{X})$. 
\end{thm}

Corollary~$\ref{cor:1.2}$ leads to an extension of Dedekind's theorem, 
while Theorem~$\ref{thm:1.3}$ leads to an inequality characterizing the degrees of irreducible representations of finite groups. 
Let $\Theta(G)$ be the group determinant of the finite group $G$, 
let $\widehat{G}$ be a complete set of inequivalent irreducible representations of $G$ over $\mathbb{C}$, 
let $\{ x_{g} \mid g \in G \}$ be independent commuting variables, 
let $\mathbb{C}[x_{g}] = \mathbb{C}\left[x_{g} ; g \in G \right]$ be the polynomial ring in $\left\{ x_{g} \: \vert \: g \in G \right\}$ with coefficients in $\mathbb{C}$, 
let $\mathbb{C} G$ be the group algebra of $G$ over $\mathbb{C}$, 
let $\mathbb{C}[x_{g}]G := \mathbb{C}[x_{g}] \otimes \mathbb{C}G = \left\{ \sum_{g \in G} c_{g} g \: \vert \: c_{g} \in \mathbb{C}[x_{g}] \right\}$, 
let $1_{G}$ be the unit element of $G$, 
and let $|G|$ be the order of $G$. 
We extend $\varphi \in \widehat{G}$ to $\varphi \colon \mathbb{C}[x_{g}]G \to \mathbb{C}[x_{g}]G$ satisfy $\varphi\left(\sum_{g \in G} c_{g} g \right) = \sum_{g \in G} c_{g} \varphi(g)$, 
where $c_{g} \in \mathbb{C}[x_{g}]$. 
The extension of Dedekind's theorem mentioned above is the following. 

\begin{thm}[Extension of Dedekind's theorem, see Theorem~$\ref{thm:9.3}$ for proof]\label{thm:1.4}
Let $G$ be a finite group and let $H$ be an abelian subgroup of $G$. 
Then, writing $\Sdet_{\Mat(1, \mathbb{C}[x_{g}]G)}^{\mathbb{C}[x_{g}]H}(\mathfrak{X}_{G})$ as $\Theta(G:H)$, we have 
\begin{align*}
\Theta(G) 1_{G} 
&= \prod_{\varphi \in \widehat{G}} \Det_{\Mat(\deg{\varphi}, \mathbb{C}[x_{g}]H)}^{\mathbb{C}[x_{g}]H} \left( \varphi(\mathfrak{X}_{G}) \right)^{\deg{\varphi}} \\ 
&= \prod_{\chi \in \widehat{H}} \chi \left( \Theta(G:H) \right) \in \mathbb{C}[x_{g}]G. 
\end{align*}
\end{thm}

The group determinant $\Theta(G) \in \mathbb{C}[x_{g}]$ is the determinant of a matrix with entries in $\{ x_{g} \mid g \in G \}$. 
(It is known that the group determinant determines the group. For the details, see \cite{Formanek_1991} and \cite{Richard}.) 
Dedekind proved the following theorem concerning the irreducible factorization of the group determinant for any finite abelian group (see, e.g., \cite{van2013history}). 

\begin{thm}[Dedekind's theorem]\label{thm:1.5}
Let $G$ be a finite abelian group. 
Then, we have
$$
\Theta(G) = \prod_{\chi \in \widehat{G}} \chi\left( \mathfrak{X}_{G} \right) \in \mathbb{C}[x_{g}]. 
$$
\end{thm}

Frobenius proved the following theorem concerning the irreducible factorization of the group determinant for any finite group; 
thus, he obtained a generalization of Dedekind's theorem (see, e.g., \cite{conrad1998origin}). 

\begin{thm}[Frobenius' theorem]\label{thm:1.6}
Let $G$ be a finite group. 
Then we have the irreducible factorization 
$$
\Theta(G) = \prod_{\varphi \in \widehat{G}} \Det_{\Mat(\deg{\varphi}, \mathbb{C}[x_{g}])}^{\mathbb{C}[x_{g}]} \left( \varphi\left( \mathfrak{X}_{G} \right) \right)^{\deg{\varphi}}. 
$$
\end{thm}

Since Dedekind's theorem is a special case of Frobenius' theorem, 
we call Frobenius' theorem a generalization of Dedekind's theorem. 
On the other hand, 
Theorem~$\ref{thm:1.4}$ gives the relation on $\mathbb{C}[x_{g}]G$, 
and Theorem~$\ref{thm:1.4}$ leads to Dedekind's theorem. 
Therefore, we call Theorem~$\ref{thm:1.4}$ an extension of Dedekind's theorem. 
That is, the `extension' is used to mean `extend $\mathbb{C}[x_{g}]$ to $\mathbb{C}[x_{g}]G$.' 

Let $H$ be an abelian subgroup of $G$ and let $[G : H]$ be the index of $H$ in $G$. 
The following extension of Dedekind's theorem, 
which is different from the theorem due to Frobenius, is given in \cite{Yamaguchi2017}. 

\begin{thm}\label{thm:1.7}
Let $G$ be a finite abelian group and let $H$ be a subgroup of $G$. 
Then, for every $h \in H$, there exists a homogeneous polynomial $c_{h} \in \mathbb{C}[x_{g}]$ such that $\deg{a_{h}} = [G : H]$ and 
$$
\Theta(G) 1_{G} = \prod_{\chi \in \widehat{H}} \sum_{h \in H} \chi(h) c_{h} h. 
$$
If $H = G$, we can take $c_{h} = x_{h}$ for each $h \in H$. 
\end{thm} 

Theorem~$\ref{thm:1.7}$ is a special case of Theorem~$\ref{thm:1.4}$. 
Theorems~$\ref{thm:1.3}$, $\ref{thm:1.4}$ and $\ref{thm:1.6}$ lead to the following corollary. 

\begin{cor}[Corollary~$\ref{cor:10.5}$]\label{cor:1.8}
Let $G$ be a finite group and let $H$ be an abelian subgroup of $G$. 
Then, for all $\varphi \in \widehat{G}$, we have 
\begin{align*}
\deg{\varphi} \leq [G:H]. 
\end{align*}
\end{cor}

Note that \ Corollary~$\ref{cor:1.8}$ follows from Frobenius reciprocity, 
and it is known that if $H$ is an abelian normal subgroup of $G$, 
then $\deg{\varphi}$ divides $[G:H]$ for all $\varphi \in \widehat{G}$ (see e.g, \cite{kondo2011group}).

This paper is organized as follows. 
In Section~$2$, we present an action of the symmetric group on the set of square matrices, 
and we introduce two formulas for determinants of commuting-block matrices. 
In addition, we recall the definition of the Kronecker product, 
and we present a permutation using the Euclidean algorithm. 
This permutation causes the order of the Kronecker product to be reversed. 
These preparations are useful for proving Theorem~$\ref{thm:1.1}$. 
In Section~$3$, 
we recall the definition of regular representations and invertibility preserving maps, 
and we show that regular representations are invertibility preserving maps. 
The regular representation is used in defining Study-type determinants. 
In addition, we formulate a commutative diagram of regular representations. 
This commutative diagram is also useful for proving Theorem~$\ref{thm:1.1}$. 
In Section~$4$, 
we prove Theorem~$\ref{thm:1.1}$. 
In Section~$5$, we give a corollary concerning the degrees of some representations contained in regular representations. 
In Section~$6$, we define Study-type determinants and elucidate their properties. 
In addition, we construct a commutative diagram for Study-type determinants. 
This commutative diagram leads to some properties of Study-type determinants. 
In Section~$7$, 
we give a Cayley-Hamilton-type theorem for the Study-type determinant under the assumption that there exists a basis of $A$ as a $B$-right module satisfying the conditions (i) and (ii). 
This Cayley-Hamilton type theorem leads to some properties of Study-type determinants. 
In Section~$8$, 
we obtain two expressions for regular representations under the assumption that the basis $e = (e_{1} \: e_{2} \: \cdots \: e_{m})$ of $A$ as a $B$-module satisfying the following conditions: 
\begin{enumerate}
\item[(iii)] for any $e_{i}$ and $e_{j}$, $e_{i}B * e_{j}B \in \left\{ e_{1}B, e_{2}B, \ldots, e_{m}B \right\}$; 
\item[(iv)] there exists $e_{k}$ such that $e_{k} B = B$; 
\item[(v)] for any $e_{i}$, there exists $e_{j}$ such that $e_{i}B * e_{j}B = B$. 
\end{enumerate}
In addition, we characterize the images of regular representations in the case that $e$ satisfies the following additional condition: 
\begin{enumerate}
\item[(vi)] for any $e_{i}$ and $e_{j}$, $e_{i}B * e_{j}B = e_{j}B * e_{i} B$. 
\end{enumerate}
This characterization is the following. 
\begin{thm}[see Theorem~$\ref{thm:7.7}$ for proof]\label{thm:1.9}
Let $L_{e}$ be the left regular representation from $\Mat(r, A)$ to $\Mat(m, \Mat(r, B))$ with respect to $e$. 
Then we have 
$$
L_{e}(A) = \{ b \in \Mat(m, B) \mid J(e_{k}) b = b J(e_{k}), k \in [m] \}. 
$$
\end{thm}
In Section~$9$, 
we introduce the Study determinant and its properties, and we derive these properties from the properties of the Study-type determinant. 
In addition, from Theorem~$\ref{thm:1.9}$, we obtain the following characterizations of $\phi_{r}$ and $\psi_{r}$. 
\begin{enumerate}
\item[(S$6'$)] $\phi_{r}(\Mat(r, \mathbb{C})) = \{ \gamma \in \Mat(2r, \mathbb{R}) \mid J_{r} \gamma = \gamma J_{r} \}$; 
\item[(S$7'$)] $\psi_{r}(\Mat(r, \mathbb{H})) = \{ \beta \in \Mat(2r, \mathbb{C}) \mid J_{r} \beta = \overline{\beta} J_{r} \}$. 
\end{enumerate}
In the last section, 
we recall the definition of group determinants, 
and we give an extension of Dedekind's theorem and derive an inequality for the degree of irreducible representations of finite groups. 

\section{Preparation}\label{Preparation}
In this section, 
we present an action of the symmetric group on the set of square matrices, 
and we introduce two formulas for determinants of commuting-block matrices. 
In addition, we recall the definition of the Kronecker product, 
and we determine a permutation using the Euclidean algorithm. 
This permutation reverses the order of the Kronecker product. 
These preparations are useful for proving Theorem~$\ref{thm:4.4}$. 

\subsection{Invariance of determinants under an action of the symmetric group}
In this subsection, we present an action of the symmetric group on the set of square matrices. 
This group action does not change the determinants of matrices. 

Let $R$ be a ring, which is assumed to have a multiplicative unit $1$, 
let $\Mat(m, R)$ be the set of all $m \times m$ matrices with elements in $R$, 
let $X = (X_{ij})_{1 \leq i, j \leq m} \in \Mat(m, R)$, 
let $[m] := \{ 1, 2, \ldots, m \}$, 
and let $S_{m}$ be the symmetric group on $[m]$. 
We express the determinant of $X$ from $\Mat(m, R)$ to $R$ as 
$$
\Det_{\Mat(m, R)}^{R}{(X)} := \sum_{\sigma \in S_{m}} \sgn(\sigma) X_{\sigma(1) \, 1} X_{\sigma(2) \, 2} \cdots X_{\sigma(m) \, m}. 
$$
The group $S_{m}$ acts on $\Mat(m, R)$ as $\sigma \cdot X := (X_{\sigma(i) \sigma(j)})_{1 \leq i, j \leq m}$, 
where $\sigma \in S_{m}$. 
If $R$ is commutative, then the group action does not change the determinants of matrices in $\Mat(m, R)$. 
In fact, we have 
\begin{align*}
\Det_{\Mat(m, R)}^{R}{(\sigma \cdot X)} 
&= \frac{1}{m!} \sum_{\tau \in S_{m}} \sum_{\tau' \in S_{m}} \sgn(\tau) \sgn(\tau') X_{\tau(\sigma(1)) \, \tau'(\sigma(1))} \cdots X_{\tau(\sigma(m)) \, \tau'(\sigma(m))} \\ 
&= \frac{1}{m!} \sum_{\tau \in S_{m}} \sum_{\tau' \in S_{m}} \sgn(\tau) \sgn(\tau') X_{\tau(1) \, \tau'(1)} \cdots X_{\tau(m) \, \tau'(m)} \\ 
&= \Det_{\Mat(m, R)}^{R}{(X)}. 
\end{align*}

\subsection{Determinants of commuting-block matrices}
In this subsection, 
we introduce two formulas for determinants of commuting-block matrices. 

Let $X = (X_{i j})_{1 \leq i, j \leq mn} \in \Mat(mn, R)$. 
The $mn \times mn$ matrix $X$ can be written as $X = \left( X^{(k, l)} \right)_{1 \leq k, l \leq n}$, 
where $X^{(k, l)}$ are $m \times m$ matrices. 
The following is a known theorem concerning commuting-block matrices \cite{ingraham1937} and \cite{doi:10.1080/00029890.1999.12005145}. 

\begin{thm}\label{thm:2.1}
Let $R$ be a commutative ring, 
and assume that $X^{(k, l)} \in \Mat(m, R)$ are commutative. 
Then we have 
\begin{align*}
\Det_{\Mat(mn, R)}^{R}{(X)} 
&= \Det_{\Mat(m, R)}^{R}{ \left( \sum_{\sigma \in S_{n}} \sgn(\sigma) X^{(1, \sigma(1))} X^{(2, \sigma(2))} \cdots X^{(n, \sigma(n))} \right) } \\ 
&= \left( \Det_{\Mat(m, R)}^{R} \circ \Det_{\Mat(mn, R)}^{\Mat(m, R)} \right) (X). 
\end{align*}
\end{thm} 

Let $I_{k}$ be the identity matrix of size $k$. 
We have the following lemma. 

\begin{lem}[]\label{lem:2.2}
Let $R'$ be a ring, 
let $R$ be a commutative ring, 
let $S$ be a subring of $R$, 
and let $\eta$ be a ring homomorphism from $R'$ to $\Mat(m, R)$. 
In this case, if $\left( \Det_{\Mat(m, R)}^{R} \circ \eta \right)(X_{ij}) \in S$ for all $1 \leq i, j \leq n$, 
then $\Det_{\Mat(mn, R)}^{R}{\left( \eta(X_{ij})_{1 \leq i, j \leq n} \right)} \in S$ holds. 
\end{lem}
\begin{proof}
(The method used here is based on that of the proof of Theorem~$\ref{thm:2.1}$ in \cite{doi:10.1080/00029890.1999.12005145}.)
We prove this by induction on $n$. 
In the case $n=1$, the statement is obviously true. 
Then, assuming that the statement is true for $n-1$, 
we prove it for $n$. 

Let $Y_{ij} = \eta(X_{ij})$ for all $1 \leq i, j \leq n$ and let $Y = (Y_{ij})_{1 \leq i, j \leq n}$. 
Then we find that the following equation holds:
\begin{align*}
&\begin{pmatrix}
I_{m} & 0 & \cdots & 0 \\ 
-Y_{2 1} & I_{m} & \cdots & 0 \\ 
\vdots & \vdots & \ddots & \vdots \\ 
-Y_{n 1} & 0 & \cdots & I_{m}
\end{pmatrix}
\begin{pmatrix}
I_{m} & 0 & \cdots & 0 \\ 
0 & Y_{1 1} & \cdots & 0 \\ 
\vdots & \vdots & \ddots & \vdots \\ 
0 & 0 & \cdots & Y_{1 1}
\end{pmatrix} 
Y
= 
\begin{pmatrix}
Y_{1 1} & * & * & * \\ 
0 & \eta(*) & \cdots & \eta(*) \\ 
\vdots & \vdots & \ddots & \vdots \\ 
0 & \eta(*) & \cdots & \eta(*) 
\end{pmatrix}. 
\end{align*}
Therefore, if $\Det_{\Mat(m, R)}^{R}(Y_{11})$ is invertible, 
then $\Det_{\Mat(mn, R)}^{R}\left( Y \right) \in S$ holds. 
Next, suppose that $\Det_{\Mat(m, R)}^{R}(Y_{11})$ is non-invertible. 
We embed $R$ in the polynomial ring $R[x]$, and replace $X_{11}$ by $x + X_{11}$. 
Then, because $\Det_{\Mat(m, R)}^{R}(\eta(x + X_{11}))$ is neither zero nor a zero divisor, 
we have $\Det_{\Mat(mn, R)}^{R}{\left( Y \right)} \in S[x]$. 
Substituting $x = 0$ yields the desired result. 
\end{proof}

\subsection{Kronecker product and a permutation obtained using the Euclidean algorithm}
In this subsection, 
we recall the definition of the Kronecker product, 
and then we determine a permutation $\sigma(m, n) \in S_{mn}$ using the Euclidean algorithm. 
This permutation reverses the order of the Kronecker product. 

Let $X = (X_{i j})_{1 \leq i \leq m_{1}, 1 \leq j \leq n_{1}}$ be an $m_{1} \times n_{1}$ matrix and let $Y$ be an $m_{2} \times n_{2}$ matrix. 
The Kronecker product $X \otimes Y$ is the $(m_{1} m_{2}) \times (n_{1} n_{2})$ matrix 
$$
X \otimes Y := 
\begin{pmatrix} 
X_{11} Y & X_{12} Y & \cdots & X_{1 n_{1}} Y \\ 
X_{21} Y & X_{22} Y & \cdots & X_{2 n_{1}} Y \\ 
\vdots & \vdots & \ddots & \vdots \\ 
X_{m_{1} 1} Y & X_{m_{1} 2} Y & \cdots & X_{m_{1} n_{1}} Y
\end{pmatrix}. 
$$

Next, we determine a permutation using the Euclidean algorithm. 
From the Euclidean algorithm, 
we know that $\sigma(m, n) : [mn] \ni m(k-1) + l \mapsto n(l-1) + k \in [mn]$ is a bijection map, 
where $k \in [n]$ and $l \in [m]$. 
Thus, $\sigma(m, n) \in S_{mn}$. 
We have the following lemma. 

\begin{lem}[]\label{lem:2.3}
Let $R$ be a commutative ring, 
let $X \in \Mat(m, R)$, 
and let $Y \in \Mat(n, R)$. 
Then we have $\sigma(m, n) \cdot (X \otimes Y) = Y \otimes X$. 
\end{lem}
\begin{proof}
For any $p, q \in [mn]$, 
by the Euclidean algorithm there exist unique integers $k, l \geq 1$ and $s, t \in [m]$ such that $p = m(k-1) + s$ and $q = m(l-1) + t$. 
Therefore, we have 
\begin{align*}
\sigma(m, n) \cdot \left( X \otimes Y \right)_{p q} 
&= \left( X \otimes Y \right)_{n(s-1) + k \: n(t-1) + l} \\ 
&= \left( X_{s t} Y \right)_{k l} \\ 
&= X_{s t} Y_{k l}. 
\end{align*}
On the other hand, we have 
\begin{align*}
\left( Y \otimes X \right)_{pq} 
&= \left( Y \otimes X \right)_{m(k-1) + s \: m(l-1) + t} \\ 
&= Y_{k l} X_{s t}. 
\end{align*}
\end{proof}

The property described by the above lemma is a special case of a property of the Kronecker product (see, e.g., \cite{doi:10.1080/03081088108817379}). 
We do not explain this general property, because, for our purposes, it is simpler to use Lemma~$\ref{lem:2.3}$.

\section{On the left regular representation}\label{On the left regular representation}
In this section, 
we recall the definition of regular representations and invertibility preserving maps, 
and we show that regular representations are invertibility preserving maps. 
A regular representation is used in defining Study-type determinants. 
In addition, we construct a commutative diagram of regular representations. 
This commutative diagram is also useful for proving Theorem~$\ref{thm:4.4}$. 

\subsection{Definition of the regular representation}
In this subsection, 
we recall the definition of regular representations, 
and we give three examples. 

Let $A$ and $B$ be rings and let $Z(A)$ be the center of $A$. 
Assume that $A$ is a free right $B$-module with an ordered basis $e = (e_{1} \: e_{2} \: \cdots \: e_{m})$. 
In other words, $A = \bigoplus_{i \in [m]} e_{i} B$, and $B$ is a subring of $A$. 
Then, for all $a \in A$, there exists a unique $(b_{i j})_{1 \leq i, j \leq m} \in \Mat(m, B)$ such that 
$$
a e_{j} = \sum_{i \in [m]} e_{i} b_{i j}. 
$$
Hence, we have $a e = e (b_{i j})_{1 \leq i, j \leq m}$. 
The injective $Z(A) \cap B$-algebra homomorphism $L_{e} : A \ni a \mapsto L_{e}(a) = (b_{i j})_{1 \leq i, j \leq m} \in \Mat(m, B)$ is called the left regular representation from $A$ to $\Mat(m, B)$ with respect to $e$.

Let $\mathbb{R}$ be the field of real numbers, 
let $\mathbb{C}$ be the field of complex numbers, 
and let $\mathbb{H} := \{ 1a + ib + jc + kd \mid a, b, c, d \in \mathbb{R} \}$ be the quaternion field. 
Below, we give three examples of regular representations. 

\begin{rei}\label{rei:3.1}
Let $A = \Mat(r, \mathbb{C})$ and let $B = \Mat(r, \mathbb{R})$. 
Then $A = B \oplus i I_{r} B$. 
For all $b_{1} + i I_{r} b_{2} \: (b_{1}, b_{2} \in B)$, 
we have 
$$
(b_{1} + i I_{r} b_{2}) (I_{r} \quad i I_{r}) = (I_{r} \quad i I_{r}) 
\begin{pmatrix}
b_{1} & - b_{2} \\ 
b_{2} & b_{1} 
\end{pmatrix}
\in \Mat(2, B). 
$$
\end{rei}

\begin{rei}\label{rei:3.2}
Let $A = \Mat(r, \mathbb{H})$ and let $B = \Mat(r, \mathbb{C})$. 
Then $A = B \oplus j I_{r} B$. 
For all $b_{1} + j I_{r} b_{2} \: (b_{1}, b_{2} \in B)$, 
we have 
$$
(b_{1} + j I_{r} b_{2}) (I_{r} \quad j I_{r}) = (I_{r} \quad j I_{r}) 
\begin{pmatrix}
b_{1} & - \overline{b_{2}} \\ 
b_{2} & \overline{b_{1}}
\end{pmatrix}
\in \Mat(2, B), 
$$
where $\overline{b}$ is the complex conjugate matrix of $b \in B$. 
\end{rei}

\begin{rei}\label{rei:3.3}
Let $G = \mathbb{Z}/ 2 \mathbb{Z} = \left\{ \overline{0}, \overline{1} \right\}$ and let $H = \left\{ \overline{0} \right\}$ be the trivial group. 
Then, the group algebra $\mathbb{C}G$ is a finite dimension algebra over $\mathbb{C}H$ with basis $\left\{ \overline{0}, \overline{1} \right\}$. 
For all $\overline{0} b_{1} + \overline{1} b_{2} \in \mathbb{C}G \: (b_{1}, b_{2} \in \mathbb{C}H)$, we have
\begin{align*}
\left( \overline{0} b_{1} + \overline{1} b_{2} \right) \left( \overline{0} \quad \overline{1} \right) = \left( \overline{0} \quad \overline{1} \right) 
\begin{pmatrix}
b_{1} & b_{2} \\ 
b_{2} & b_{1}
\end{pmatrix}
\in \Mat(2, \mathbb{C}H). 
\end{align*}
\end{rei}

\subsection{Definition of the invertibility preserving map}
In this subsection, we recall the definition of invertibility preserving maps, 
and we show that regular representations are invertibility preserving maps. 
Usually, invertibility preserving maps are defined for linear maps (see, e.g., \cite{bresar1998}). 
However, we do not assume that invertibility preserving maps are linear maps as in \cite{yamaguchi2016compositions}. 

The following is the definition of invertibility preserving maps. 

\begin{definition}[Invertibility preserving map]\label{def:3.4}
Let $R$ and $R'$ be rings, 
and let $\eta : R \rightarrow R'$ be a map. 
Assume that for any $\alpha \in R$, 
the following condition holds: 
$\alpha$ is invertible in $R$ if and only if $\eta(\alpha)$ is invertible in $R'$. 
Then we call $\eta$ an ``invertibility preserving map.'' 
\end{definition} 

We recall that if $B$ is a commutative ring, 
then we do not need to distinguish between left and right inverses for $a \in A$. 
Because, $L_{e}$ is an injective algebra homomorphism, 
and if $L_{e}(a) L_{e}(b)$ is the unit element, 
then $L_{e}(b) L_{e}(a)$ is the unit element. 

We denote the unit element of $A$ as $1$. 
In terms of the regular representation, we have the following lemma. 

\begin{lem}\label{lem:3.5}
If $B$ is a commutative ring, then we have the following properties: 
\begin{enumerate}
\item the map $\Det_{\Mat(m, B)}^{B}{} \circ L_{e}$ is invariant under a change of the basis $e$; 
\item a left regular representation is an invertibility preserving map. 
\end{enumerate}
\end{lem} 
\begin{proof}
First, we prove $(1)$. 
Let $L_{e'}$ be another left regular representation from $A$ to $\Mat(m, B)$. 
Then for all $a \in A$, there exists $Q \in \Mat(m, B)$ such that $L_{e}(a) = Q^{-1} L_{e'}(a) Q$. 
Therefore, we have $\Det_{\Mat(m, B)}^{B} \circ L_{e} = \Det_{\Mat(m, B)}^{B} \circ L_{e'}$. 
Next, we prove $(2)$. 
If $a$ is invertible in $A$, 
then we have $L_{e}(a a^{-1}) = I_{m}$. 
Note that because $L_{e}$ is a multiplicative map, $L_{e}(a)$ is invertible. 
Conversely, if $\left( \Det_{\Mat(m, B)}^{B} \circ L_{e} \right)(a)$ is invertible, 
then $L_{e}(a)$ is invertible. 
Because $1 \in A$, we can choose $e_{1} = 1$. 
Then, we have $a e L_{e}(a)^{-1} = e I_{m}$. 
Therefore, we obtain $a \left( \sum_{i \in [m]} e_{i} \left( L_{e}(a)^{-1} \right)_{i1} \right) = 1$. 
Hence, $a$ is invertible. 
This completes the proof. 
\end{proof}

\subsection{Commutative diagram of regular representations}
In this subsection, 
we present a commutative diagram of regular representations. 


Let $\mathbb{N} := \{1, 2, \ldots \}$ be the set of natural numbers, 
let $C$ be a ring, 
and let $B$ be a free right $C$-module with an ordered basis $f = ( f_{1} \: f_{2} \: \cdots \: f_{n})$. 
Then we have the direct sum 
\begin{align*}
\Mat(r, B) = \bigoplus_{i \in [n]} (f_{i} \otimes I_{r}) \Mat(r, C) 
\end{align*}
for any $r \in \mathbb{N}$. 
We write the left regular representation from $B$ to $\Mat(n, C)$ with respect to $f$ as $L_{f}$ and that from $\Mat(m, B)$ to $\Mat(n, \Mat(m, C))$ with respect to $f \otimes I_{m}$ as $L_{f \otimes I_{m}}$. 
In terms of the regular representations, we have the following lemma. 
\begin{lem}\label{lem:3.6}
The following diagram is commutative: 
\[
\xymatrix{
A \ar[r]^-{L_{e}} \ar[d]_-{L_{e \otimes f}} & \Mat(m, B) \ar[d]^-{L_{f \otimes I_{m}}} \\ 
\Mat(mn, C) \ar@{^{(}-_>}[r]^-{} & \Mat(n, \Mat(m, C)) \\ 
}
\]
\end{lem} 
\begin{proof}
For all $a \in A$, we have
\begin{align*}
a (e \otimes f) 
&= a e (f \otimes I_{m}) \\ 
&= e L_{e}(a) (f \otimes I_{m}) \\ 
&= e (f \otimes I_{m}) L_{f \otimes I_{m}}(L_{e}(a)) \\ 
&= (e \otimes f) (L_{f \otimes I_{m}} \circ L_{e})(a). 
\end{align*}
This completes the proof. 
\end{proof}

Let $r \in \mathbb{N}$ and let $L_{e \otimes I_{r}}$ be the left regular representation from $\Mat(r, A)$ to $\Mat(m, \Mat(r, B))$ with respect to $e \otimes I_{r}$. 
Then, from Lemma~$\ref{lem:3.6}$, we have the following corollary.

\begin{cor}\label{cor:3.7}
The following diagram is commutative: 
\[
\xymatrix{
\Mat(r, A) \ar[r]^-{L_{e \otimes I_{r}}} \ar[d]_-{L_{(e \otimes f) \otimes I_{r}}} & \Mat(m, \Mat(r, B)) \ar[d]^-{L_{(f \otimes I_{r}) \otimes I_{m}}} \\ 
\Mat(mn, \Mat(r, C)) \ar@{^{(}-_>}[r]^-{} & \Mat(n, \Mat(m, \Mat(r, C))) \\ 
}
\]
\end{cor}

\section{A commutative diagram on the regular representations and determinants}
In this section, we prove Theorem~$\ref{thm:4.4}$. 
This theorem provides a commutative diagram on the regular representations and determinants. 
From this commutative diagram, we are able to determine the properties of Study-type determinants, presented in Section~$6$. 
In addition, from this commutative diagram, we are able to derive an inequality for the degrees of representations (Section~$5$) and an extension of Dedekind's theorem (Section~$10$). 

Let $E_{ij}$ be the $r \times r$ matrix with $1$ in the $(i, j)$ entry and $0$ otherwise. 
First, we prove the following lemma: 

\begin{lem}\label{lem:4.1}
For any $a = (a_{ij})_{1 \leq i, j \leq r} \in \Mat(r, A)$, 
we have 
$$
L_{e \otimes I_{r}}(a) = \sum_{i \in [r]} \sum_{j \in [r]} L_{e}(a_{ij}) \otimes E_{ij}. 
$$
\end{lem} 
\begin{proof}
We express $\left( L_{e}(a_{i j})_{k l} \right)_{1 \leq i, j \leq r}$ as $a_{e}^{(k, l)}$ for any $a = (a_{ij})_{1 \leq i, j \leq r} \in \Mat(r, A)$. 
From $a_{ij} e_{l} = \sum_{k \in [m]} e_{k} L_{e}(a_{ij})_{kl}$, 
we obtain 
\begin{align*}
\sum_{k \in [m]} e_{k} I_{r} a_{e}^{(k, l)} 
&= \sum_{k \in [m]} \left( e_{k} L_{e}(a_{ij})_{kl} \right)_{1 \leq i, j \leq r} \\ 
&= (a_{ij} e_{l})_{1 \leq i, j \leq r} \\ 
&= a e_{l} I_{r}. 
\end{align*}
Therefore, we have 
$(e \otimes I_{r}) \left( a_{e}^{(k, l)} \right)_{1 \leq k, l \leq m} = a (e \otimes I_{r})$. 
This completes the proof. 
\end{proof}

From Lemmas~$\ref{lem:2.3}$ and $\ref{lem:4.1}$, we obtain the following lemma. 

\begin{lem}\label{lem:4.2}
For any $a = (a_{ij})_{1 \leq i, j \leq r} \in \Mat(r, A)$, 
we have 
$$
\sigma(m, r) \cdot L_{e \otimes I_{r}}(a) = \sum_{i \in [r]} \sum_{j \in [r]} E_{ij} \otimes L_{e}(a_{ij}) = \left( L_{e}(a_{i j}) \right)_{1 \leq i, j \leq r}. 
$$
\end{lem}

Next, from Lemmas~$\ref{lem:2.2}$ and $\ref{lem:4.2}$, 
we obtain the following corollary. 

\begin{cor}\label{cor:4.3}
Let $B$ be a commutative ring, 
let $A$ be a ring that is a free right $B$-module, 
let $S$ be a subring of $B$, 
and let $L$ and $L'$ be left regular representations from $A$ to $\Mat(m, B)$ and from $\Mat(r, A)$ to $\Mat(m, \Mat(r, B))$, respectively. 
If $\left( \Det_{\Mat(m, B)}^{B} \circ L \right)(a) \in S$ for all $a \in A$, 
then $\left( \Det_{\Mat(mr, B)}^{B} \circ L' \right)(a') \in S$ for all $a' \in \Mat(r, A)$. 
\end{cor}

In the following, we assume that $B$ and $C$ are commutative rings. 
Then we have the following theorem. 

\begin{thm}[Theorem~$\ref{thm:1.1}$]\label{thm:4.4}
Let $B$ and $C$ be commutative rings, 
let $B$ be a free right $C$-module, 
and let $L$ and $L'$ be left regular representations from $B$ to $\Mat(n, C)$ and from $\Mat(r, B)$ to $\Mat(n, \Mat(r, C))$, respectively. 
Then the following diagram is commutative: 
\[
\xymatrix{
\Mat(r, B) \ar[d]^-{L'} \ar[r]^-{\Det_{\Mat(r, B)}^{B}} & B \ar[r]^-{L} & \Mat(n, C) \ar[d]^-{\Det_{\Mat(n, C)}^{C}} \\ 
\Mat(n, \Mat(r, C)) \ar@{^{(}-_>}[r]^-{} & \Mat(nr, C) \ar[r]^-{\Det_{\Mat(nr, C)}^{C}} & C \\ 
 &  
}
\]
\end{thm}
\begin{proof}
Without loss of generality, 
we can assume that $L = L_{f}$ and $L' = L_{f \otimes I_{r}}$. 
Let $b = (b_{ij})_{1 \leq i, j \leq r} \in \Mat(r, B)$. 
Then, $L_{f}$ is a ring homomorphism and $B$ is a commutative ring, from Theorem~$\ref{thm:2.1}$ and Lemma $\ref{lem:4.2}$, 
we have 
\begin{align*}
\left( \Det_{\Mat(nr, C)}^{C} \circ L_{f \otimes I_{r}} \right)(b) 
&= \left(\Det_{\Mat(nr, C)}^{C} \circ \sigma(n, r) \cdot L_{f \otimes I_{r}} \right)(b) \\ 
&= \Det_{\Mat(nr, C)}^{C}{\left( (L_{f}(b_{i j}))_{1 \leq i, j \leq r} \right)} \\ 
&= \Det_{\Mat(n, C)}^{C}{\left( \sum_{\sigma \in S_{r}} \sgn(\sigma) L_{f}(b_{1 \: \sigma(1)}) \cdots L_{f}(b_{r \: \sigma(r)}) \right)} \\ 
&= \left( \Det_{\Mat(n, C)}^{C} \circ L_{f} \right) \left( \sum_{\sigma \in S_{r}} \sgn(\sigma) b_{1 \: \sigma(1)} \cdots b_{r \: \sigma(r)} \right) \\ 
&= \left( \Det_{\Mat(n, C)}^{C} \circ L_{f} \circ \Det_{\Mat(r, B)}^{B} \right) (b)
\end{align*}
This completes the proof. 
\end{proof}

\section{Degrees of some representations contained in regular representations}
In this section, we give a corollary regarding the degrees of some representations contained in regular representations.

The following is the definition of the general element. 

\begin{definition}[General element]
Let $S$ be a finite set and let $[x_{s}] := \{ x_{s} \mid s \in S \}$ be a set of independent commuting variables. 
The general element for $S$ defined as 
$$
\mathfrak{X}_{S} := \sum_{s \in S} s x_{s} \in S[x_{s}], 
$$
where $S [x_{s}]$ is the set of polynomials $[x_{s}]$ over $S$. 
\end{definition}

Let $ef := \{ e_{i} f_{j} \mid i \in [m], j \in [n] \}$ and let $[x_{\alpha}] := \{ x_{\alpha} \mid \alpha \in ef \}$ be the set of independent commuting variables, 
and we denote $\mathfrak{X}_{ef} \in (ef)[x_{\alpha}] \subset A[x_{\alpha}]$ as $\mathfrak{X}$. 
For rings $R$ and $R'$, 
we denote the set of ring homomorphisms from $R$ to $R'$ by $\Hom(R, R')$, 
and we regard any ring homomorphism $\rho \in \Hom(R, R')$ as $\rho \in \Hom(R[x_{\alpha}], R'[x_{\alpha}])$ such that $\rho(x_{\alpha} 1_{R}) = x_{\alpha} \rho(1_{R})$ for any $\alpha \in ef$, 
where $1_{R}$ is the unit element of $R$. 

From Theorem~$\ref{thm:4.4}$, for any $\rho \in \Hom(A, \Mat(r, B))$, we have 
$$
\left( \Det_{\Mat(nr, C[x_{\alpha}])}^{C[x_{\alpha}]} \circ L_{f \otimes I_{r}} \circ \rho \right) (\mathfrak{X}) = \left( \Det_{\Mat(n, C[x_{\alpha}])}^{C[x_{\alpha}]} \circ L_{f} \circ \Det_{\Mat(r, B[x_{\alpha}])}^{B[x_{\alpha}]} \circ \rho \right)(\mathfrak{X}). 
$$
Let $\bar{C}$ be a commutative ring such that $C \subset \bar{C}$. 
We assume that $L_{f \otimes I_{r}} \circ \rho$ and $L_{f}$ have the following direct sums: 
\begin{align*}
L_{f \otimes I_{r}} \circ \rho &\sim \varphi^{(1)} \oplus \varphi^{(2)} \oplus \cdots \oplus \varphi^{(s)}, &&\varphi^{(i)}(a) \in \Mat(r_{i}, \bar{C}), \\ 
L_{f} &\sim \psi^{(1)} \oplus \psi^{(2)} \oplus \cdots \oplus \psi^{(t)}, &&\psi^{(j)}(b) \in \Mat(n_{j}, \bar{C}), 
\end{align*}
where $a \in A$ and $b \in B$. 
Then we have the following corollary from Theorem~$\ref{thm:4.4}$. 

\begin{cor}[Corollary~$\ref{cor:1.2}$]\label{cor:5.1}
The following hold: 
\begin{align*}
\prod_{1 \leq i \leq s} \left( \Det_{\Mat(r_{i}, \bar{C}[x_{\alpha}])}^{\bar{C}[x_{\alpha}]} \circ \varphi^{(i)} \right) (\mathfrak{X}) 
&= \left( \Det_{\Mat(nr, C[x_{\alpha}])}^{C[x_{\alpha}]} \circ L_{f \otimes I_{r}} \circ \rho \right) (\mathfrak{X}) \\ 
&= \left( \Det_{\Mat(n, C[x_{\alpha}])}^{C[x_{\alpha}]} \circ L_{f} \circ \Det_{\Mat(r, B[x_{\alpha}])}^{B[x_{\alpha}]} \circ \rho \right) (\mathfrak{X}) \\
&= \prod_{1 \leq j \leq t} \left( \Det_{\Mat(n_{j}, \bar{C}[x_{\alpha}])}^{\bar{C}[x_{\alpha}]} \circ \psi^{(j)} \circ \Det_{\Mat(r, B[x_{\alpha}])}^{B[x_{\alpha}]} \circ \rho \right) (\mathfrak{X}). 
\end{align*}
\end{cor}

Corollary~$\ref{cor:5.1}$ leads to the following theorem. 

\begin{thm}[Theorem~$\ref{thm:1.3}$]\label{thm:5.3}
If $\left( \Det_{\Mat(r_{i}, \bar{C}[x_{\alpha}])}^{\bar{C}[x_{\alpha}]} \circ \varphi^{(i)} \right) (\mathfrak{X})$ is an irreducible polynomial over $\bar{C}[x_{\alpha}]$, 
then we have 
$$
\deg{\varphi^{(i)}} \leq \max{ \left\{ \deg{\psi^{(j)}} \mid j \in [t] \right\} } \times \deg{\rho}, 
$$
where $\deg{\rho}$ is the degree of the polynomial $\left( \Det_{\Mat(r, B[x_{\alpha}])}^{B[x_{\alpha}]} \circ \rho \right)(\mathfrak{X})$. 
\end{thm}
\begin{proof}
If $\left( \Det_{\Mat(r_{i}, \bar{C}[x_{\alpha}])}^{\bar{C}[x_{\alpha}]} \circ \varphi^{(i)} \right) (\mathfrak{X})$ is an irreducible polynomial over $\bar{C}[x_{\alpha}]$, 
then we have
\begin{align*}
\deg{\varphi^{(i)}} 
&= \deg{\left( \left( \Det_{\Mat(r_{i}, \bar{C}[x_{\alpha}])}^{\bar{C}[x_{\alpha}]} \circ \varphi^{(i)} \right) (\mathfrak{X}) \right)} \\ 
&\leq \max{ \left\{ \deg{\left( \Det_{\Mat(n_{j}, \bar{C}[x_{\alpha}])}^{\bar{C}[x_{\alpha}]} \circ \psi^{(j)} \circ \Det_{\Mat(r, B)}^{B} \circ \rho \right) (\mathfrak{X})} \mid j \in [t] \right\} } \\ 
&= \max{ \left\{ \deg{ \psi^{(j)} } \times \deg{ \left( \Det_{\Mat(r, B[x_{\alpha}])}^{B[x_{\alpha}]} \circ \rho \right)(\mathfrak{X}) } \mid j \in [t] \right\} } \\ 
&= \max{ \left\{ \deg{\psi^{(j)}} \mid j \in [t] \right\} } \times \deg{\rho}. 
\end{align*}
This completes the proof. 
\end{proof}

\section{On the Study-type determinant}
In this section, we define the Study-type determinant, 
and we elucidate its properties. 
The Study-type determinant is a generalization of the Study determinant. 
In addition, we present a commutative diagram characterizing Study-type determinants. 
This commutative diagram allows us to determine some properties of the Study-type determinant.

The following is the definition of the Study-type determinant. 

\begin{definition}[Study-type determinant]\label{def:6.1}
Let $B$ be a commutative ring, 
let $A$ be a ring that is a free right $B$-module, 
and let $L$ be a left regular representation from $\Mat(r, A)$ to $\Mat(m, \Mat(r, B))$. 
We define the Study-type determinant $\Sdet_{\Mat(r, A)}^{B}$ as 
$$
\Sdet_{\Mat(r, A)}^{B} := \Det_{\Mat(mr, B)}^{B} \circ \iota \circ L, 
$$
where $\iota$ is the inclusion map from $\Mat(m, \Mat(r, B))$ to $\Mat(mr, B)$. 
\end{definition}

A Study-type determinant is a multiplicative and invertibility preserving map, 
because determinants and left regular representations are multiplicative and invertibility preserving maps. 
Thus, we have the following lemma. 

\begin{lem}[(S$1$) and (S$2$)]\label{lem:6.2}
Study-type determinants possess the following properties: 
\begin{enumerate}
\item a Study-type determinant is a multiplicative map; 
\item a Study-type determinant is an invertibility preserving map. 
\end{enumerate}
\end{lem} 

In addition, we have the following lemma. 

\begin{lem}[(S$3$)]\label{lem:6.3}
If $a' \in \Mat(r, A)$ is obtained from $a \in \Mat(r, A)$ by adding a left-multiple of a row to another row or a right-multiple of a column to another column, 
then we have $\Sdet_{\Mat(r, A)}^{B}(a') = \Sdet_{\Mat(r, A)}^{B}(a)$. 
\end{lem}
\begin{proof}
The property~(S$1$) can be restated as $\Sdet_{\Mat(r, A)}^{B}(I_{r} + aE_{ij}) = 1$, where $a \in A$. 
Then, from Theorem~$\ref{thm:2.1}$ and Lemma~$\ref{lem:4.1}$, 
we have 
\begin{align*}
\Sdet_{\Mat(r, A)}^{B}(I_{r} + a E_{ij}) 
&= \Det_{\Mat(mr, A)}^{B} \left( I_{mr} + E_{ij} \otimes L_{e}(a) \right) = 1. 
\end{align*}
This completes the proof. 
\end{proof}

For a ring $R$ and $k \in \mathbb{N}$, we denote $\Mat(k, R)$ as $\Mat_{k}(R)$. 
From Corollary~$\ref{cor:3.7}$ and Theorem~$\ref{thm:4.4}$, 
we obtain the following commutative diagram: 
\[
\xymatrix{
\Mat_{r}(A) \ar[r]^-{L_{e \otimes I_{r}}} \ar[d]^-{L_{(e \otimes f) \otimes I_{r}}} & \Mat_{m}(M_{r}(B)) \ar[d]^-{L_{(f \otimes I_{r}) \otimes I_{m}}} \ar@{^{(}-_>}[r]^-{} & \Mat_{mr}(B) \ar[d]^-{L_{f \otimes I_{mr}}} \ar[r]^{\Det_{\Mat_{mr}(B)}^{B}} & B \ar[r]^-{L_{f}} & \Mat_{n}(C) \ar[d]_-{\Det_{\Mat_{n}(C)}^{C}} \\ 
\Mat_{mn}(\Mat_{r}(C)) \ar@{^{(}-_>}[r]^-{} & \Mat_{n}(\Mat_{m}(\Mat_{r}(C))) \ar@{^{(}-_>}[r]^-{} & \Mat_{n}(\Mat_{mr}(C)) \ar@{^{(}-_>}[r]^-{} & \Mat_{mnr}(C) \ar[r]_-{\Det_{\Mat_{mnr}(C)}^{C}} & C \\ 
}
\]
From this, we obtain the following theorem. 

\begin{thm}[(S$5$)]\label{thm:6.4}
Let $B$ and $C$ be commutative rings, 
let $A$ be a ring that is a free right $B$-module, 
and let $B$ be a free right $C$-module. 
Then we have 
$$
\Sdet_{\Mat(r, A)}^{C} = \Sdet_{\Mat(1, B)}^{C} \circ \Sdet_{\Mat(r, A)}^{B}, 
$$
where we regard $B$ as $\Mat(1, B)$. 
\end{thm}

Lemmas~$\ref{lem:6.2}$ and $\ref{lem:6.3}$ and Theorem~$\ref{thm:6.4}$ are equal to (S1), (S2), (S3), and (S5) in Section~$1$.

\section{Characteristic polynomial and Cayley-Hamilton-type theorem for the Study-type determinant}\label{Characteristic polynomial}
In this section, 
we give a Cayley-Hamilton-type theorem for the Study-type determinant under the assumption that there exists a basis of $A$ as $B$-module satisfying certain conditions. 
This Cayley-Hamilton-type theorem leads to some properties of Study-type determinants. 

Let $L$ be a left regular representation from $A$ to $\Mat(m, B)$ and let $x$ be an independent variable. 
We write $(\Det_{A}^{B} \circ \iota)(x I_{m} - L(a)) \in B[x]$ as $\Phi_{L(a)}(x)$ for all $a \in A$, 
i.e., we express the characteristic polynomial of $L(a)$ as $\Phi_{L(a)}(x)$. 
In this section, we assume that $e$ has the following properties: 
\begin{enumerate}
\item[(i)] $e_{i}$ is invertible in $A$ for any $e_{i}$;  
\item[(ii)] $e_{i}^{-1} B e_{i} \subset B$ holds for any $e_{i}$, 
\end{enumerate}
Then we have the following lemma. 

\begin{lem}\label{lem:7.1}
We have $\Phi_{L(\alpha)}(x) \in (Z(A) \cap B)[x]$. 
In particular, $\Sdet_{A}^{B}{(\alpha)} \in Z(A) \cap B$. 
\end{lem} 
\begin{proof}
Without loss of generality, we can assume that $L = L_{e}$. 
We show that $e_{i}^{-1} \Phi_{L(a)}(x) e_{i} = \Phi_{L(a)}(x)$ for all $e_{i}$. 
Since $e_{i}$ is invertible for all $e_{i}$, 
there exists an invertible element $Q_{i} \in \Mat(m, B)$ such that $e e_{i} = e Q_{i}$. 
Then, from $a e = e e_{i} e_{i}^{-1} L(a) e_{i} e_{i}^{-1}$, 
we obtain $a e Q_{i} = e Q_{i} e_{i}^{-1} L(a) e_{i}$. 
Therefore, we have $a e = e Q_{i} e_{i}^{-1} L(a) e_{i} Q_{i}^{-1}$. 
Also, because $L$ is injective and $e_{i}^{-1} L(a) e_{i} \in \Mat(m, B)$, 
we have $Q_{i}^{-1} L(a) Q_{i} = e_{i}^{-1} L(a) e_{i}$. 
Therefore, we have
\begin{align*}
e_{i}^{-1} \Phi_{L(a)}(x) e_{i} 
&= \Det_{\Mat(m, B)}^{B} \left( x I_{m} - e_{i}^{-1} L(a) e_{i} \right) \\ 
&= \Det_{\Mat(m, B)}^{B} \left( x I_{m} - Q_{i}^{-1} L(a) Q_{i} \right) \\ 
&= \Phi_{L(a)}(x). \\ 
\end{align*}
This completes the proof. 
\end{proof}

The following theorem is a Cayley-Hamilton-type theorem. 

\begin{thm}[Cayley-Hamilton-type theorem]\label{thm:7.2}
Let $\Phi_{L(a)}(x) = x^{m} + b_{m-1} x^{m-1} + \cdots + b_{0}$ be the characteristic polynomial of $L(a)$. 
Then we have 
$$
\Phi_{L(a)}(a) = a^{m} + b_{m-1} a^{m-1} + \cdots + b_{0} = 0. 
$$
\end{thm}
\begin{proof}
Without loss of generality, we can assume that $L = L_{e}$. 
From the Cayley-Hamilton theorem for commutative rings, 
we have $L(a)^{m} + b_{m-1} L(a)^{m-1} + \cdots + b_{0} I_{m} = 0$. 
Then, because $L$ is a $Z(A) \cap B$-algebra homomorphism, from Lemma~$\ref{lem:7.1}$, 
we obtain 
$$
L(a)^{m} + b_{m-1} L(a)^{m-1} + \cdots + b_{0} I_{m} = L(a^{m} + b_{m-1} a^{m-1} + \cdots + b_{0}) = 0. 
$$
Finally, because $L$ is injective, we have $\Phi_{L(a)}(a) = 0$. 
This completes the proof. 
\end{proof}

From Corollary~$\ref{cor:4.3}$ and Lemma~$\ref{lem:7.1}$, 
we obtain the following corollary. 

\begin{cor}[(S$4$)]\label{cor:7.3}
For all $a \in \Mat(r, A)$, we have $\Sdet_{\Mat(r, A)}^{B}(a) \in Z(A) \cap B$. 
\end{cor}

Next, from Corollaries~$\ref{cor:4.3}$ and $\ref{cor:7.3}$, 
we obtain the following corollary. 

\begin{cor}[(S$6$)]\label{cor:7.4} 
For all $a \in \Mat(r, A)$, we have $\Sdet_{\Mat(r, A)}^{C}(a) = \left( \Sdet_{\Mat(r, A)}^{B}(a) \right)^{n}$. 
\end{cor}

\section{Image of a regular representation when the direct sum forms a group}
In this section, 
we obtain two expressions for regular representations and we characterize the image of a regular representation in the case that a basis of $A$ as $B$-module satisfies certain conditions. 

Let $e B := \{ e_{i} B \mid i \in [m] \}$ and we define a product of $e_{i} B$ and $e_{j} B$ as 
$$
e_{i} B * e_{j} B := \left\{ e_{i} e_{j} b \mid b \in B \right\}. 
$$
In this section, we assume that $e$ satisfies the following conditions: 
\begin{enumerate}
\item[(iii)] for any $e_{i}$ and $e_{j}$, $e_{i}B * e_{j}B \in eB$; 
\item[(iv)] there exists $e_{k}$ such that $e_{k} B = B$ as set; 
\item[(v)] for any $e_{i}$, there exists $e_{j}$ such that $e_{i}B * e_{j}B = B$. 
\end{enumerate}
It is easy to show that $e_{i}$ is invertible in $A$ and $(eB, *)$ is a group. 
For a group $G$, we denote the unit element of $G$ as $1_{G}$. 
We remark that even if $e'$ has the above properties, 
then $eB$ and $e'B$ are not necessary group isomorphism. 

\begin{rei}
Let $G = \mathbb{Z} / 2 \mathbb{Z} \times \mathbb{Z} / 2 \mathbb{Z}$ and $\alpha := \frac{1-i}{2} \left( \overline{1}, \overline{0} \right) + \frac{1 + i}{2} \left( \overline{0}, \overline{1} \right) \in \mathbb{C} G$. 
Then, $\left\{ \left( \overline{0}, \overline{0} \right), \left( \overline{1}, \overline{0} \right), \left( \overline{0}, \overline{1} \right), \left( \overline{1}, \overline{1} \right) \right\}$ and $\left\{ 1_{G}, \alpha, \alpha^{2}, \alpha^{3} \right\}$ are groups and basis of $\mathbb{C}G$ as a $\mathbb{C}$-right module, respectively. 
However, $\left\{ 1_{G}, \left( \overline{1}, \overline{0} \right), \left( \overline{0}, \overline{1} \right), \left( \overline{1}, \overline{1} \right) \right\} \ncong \left\{ 1_{G}, \alpha, \alpha^{2}, \alpha^{3} \right\}$ as group. 
\end{rei}

Let $P(e)$ be the diagonal matrix $\diag{(e_{1}, e_{2}, \ldots, e_{m})} \in \Mat(m, A)$. 
To obtain an expression for $L_{e}$, 
we define the indicator function ${\bf 1}_{B}$ by 
\begin{align*}
{\bf 1}_{B}(b) := 
\begin{cases}
1, & b \in B, \\ 
0, & b \not\in B. 
\end{cases}
\end{align*}
We now formulate an expression for regular representations in terms of ${\bf 1}_{B}$. 

\begin{lem}\label{lem:7.5}
Let $a = \sum_{k \in [m]} e_{k} b_{k} \in A$, where $b_{k} \in B$. 
Then we have 
$$
L_{e}(a)_{i j} = \sum_{k \in [m]} {\bf 1}_{B}(e_{i}^{-1} e_{k} e_{j}) e_{i}^{-1} e_{k} b_{k} e_{j}. 
$$
\end{lem} 
\begin{proof}
For all $a = \sum_{k \in [m]} e_{k} b_{k} \in A$, we have 
\begin{align*}
&e \left( \sum_{k \in [m]} {\bf 1}_{B}(e_{i}^{-1} e_{k} e_{j}) e_{i}^{-1} e_{k} b_{k} e_{j} \right)_{1 \leq i, j \leq m} \\ 
&\quad = \left( \sum_{i \in [m]} \sum_{k \in [m]} {\bf 1}_{B}(e_{i}^{-1} e_{k} e_{1}) e_{k} b_{k} e_{1} \quad \cdots \quad \sum_{i \in [m]} \sum_{k \in [m]} {\bf 1}_{B}(e_{i}^{-1} e_{k} e_{m}) e_{k} b_{k} e_{m} \right) \\ 
&\quad = \left( \sum_{k \in [m]} e_{k} b_{k} e_{1} \quad \cdots \quad \sum_{k \in [m]} e_{k} b_{k} e_{m} \right) \\ 
&\quad = a e. 
\end{align*}
This completes the proof. 
\end{proof}

Let $L_{eB}$ be the regular representation of the group $eB$. 
From 
$$
L_{eB}(e_{k} B)_{i j} = \left( {\bf 1}_{B}(e_{i}^{-1} e_{k} e_{j}) \right)_{1 \leq i, j \leq m}, 
$$
we can obtain an another expression for regular representations. 

\begin{cor}\label{cor:7.6}
Let $a = \sum_{k \in [m]} e_{k} b_{k} \in A$, where $b_{k} \in B$. 
Then we have 
$$
L_{e}(a) = P(e)^{-1} \left( \sum_{k \in [m]} L_{eB}(e_{k} B) e_{k} b_{k} \right) P(e). 
$$
\end{cor}

We add the following assumption: 
\begin{enumerate}
\item[(vi)] for any $e_{i}B$ and $e_{j}B$, $e_{i}B * e_{j} B = e_{j} B * e_{i} B$. 
\end{enumerate}
Let $J(e_{k}) := P(e)^{-1} \left( L_{eB}(e_{k} B) \right) P(e)$ for all $k \in [m]$, 
and we write $\{ \rho(s) \mid s \in S \}$ as $\rho(S)$ for any map $\rho$ and any set $S$. 
We show that $b \in \Mat(m, B)$ is an image of $L_{e}$ if and only if $b$ commutes with $J(e_{k})$ for all $k \in [m]$. 

\begin{thm}[Theorem~$\ref{thm:1.9}$]\label{thm:7.7}
We have 
$$
L_{e}(A) = \{ b \in \Mat(m, B) \mid J(e_{k}) b = b J(e_{k}), k \in [m] \}. 
$$
\end{thm} 
\begin{proof}
From Corollary~$\ref{cor:7.6}$, we have $L_{e}(A) \subset \{ b \in \Mat(m, B) \mid J(e_{k}) b = b J(e_{k}), k \in [m] \}$. 
We show that $\{ b \in \Mat(m, B) \mid J(e_{k}) b = b J(e_{k}), k \in [m] \} \subset L_{e}(A)$. 
For all $b \in \Mat(m, B)$, there exists $a \in A$ and $b'_{ij} \in B$ such that $b = L_{e}(a) + b'$ where $b' = (b'_{ij})_{1 \leq i, j \leq m}$ and $b'_{k1} = 0$ for all $k \in [m]$. 
Also from Corollary~$\ref{cor:7.6}$, we have $b' J(e_{k}) = J(e_{k}) b'$ for all $k \in [m]$. 
Further, we know that for all $j \in [m] \setminus \{1\}$, there exists $e_{k}$ such that $(J(e_{k}))_{1 1} = e_{j} e_{1}^{-1}$ and $(J(e_{k}))_{i 1} = 0$ for all $i \neq j$. 
Therefore, we have 
\begin{align*}
b'_{i j} e_{j} e_{1}^{-1} 
&= b'_{i j} (J(e_{k}))_{j 1} \\ 
&= (b' J(e_{k}))_{i 1} \\ 
&= (J(e_{k}) b')_{i 1} \\ 
&= 0. 
\end{align*}
This implies $b = L_{e}(a) \in L_{e}(A)$. 
This completes the proof. 
\end{proof}

From Lemma~$\ref{lem:4.1}$ and Theorem~$\ref{thm:7.7}$, we obtain the following corollary. 

\begin{cor}\label{cor:7.8}
We have 
$$
L_{e \otimes I_{r}}(\Mat(r, A)) = \{ b \in \Mat(mr, B) \mid (J(e_{k}) \otimes I_{r}) b = b J(e_{k}) \otimes I_{r}), k \in [m] \}. 
$$
\end{cor}

\section{On the relationship between the Study-type and Study determinants}
In this section, 
we introduce the Study determinant and elucidate its properties. 
We derive these properties from the properties of the Study-type determinant. 

First, we recall that any complex $r \times r$ matrix can be written uniquely as $b = c_{1} + i c_{2}$, where $c_{1}$ and $c_{2} \in \Mat(r, \mathbb{R})$, 
and any quaternionic $r \times r$ matrix can be written uniquely as $a = b_{1} + j b_{2}$, where $b_{1}$ and $b_{2} \in \Mat(r, \mathbb{C})$. 
We define $\phi_{r} : \Mat(r, \mathbb{C}) \rightarrow \Mat(2r, \mathbb{R})$ and $\psi_{r} : \Mat(r, \mathbb{H}) \rightarrow \Mat(2r, \mathbb{C})$ by 
\begin{align*}
\phi_{r}(c_{1} + i c_{2}) := 
\begin{pmatrix}
c_{1} & - c_{2} \\ 
c_{2} & c_{1} 
\end{pmatrix}, \quad 
\psi_{r}(b_{1} + j b_{2}) := 
\begin{pmatrix}
b_{1} & -\overline{b_{2}} \\ 
b_{2} & \overline{b_{1}} 
\end{pmatrix}, 
\end{align*}
respectively. 
The Study determinant $\Sdet{}$ is defined by 
$$
\Sdet{(a)} := \left( \Det_{\Mat(2r, \mathbb{C})}^{\mathbb{C}} \circ \psi_{r} \right)(a) 
$$ 
for all $a \in \Mat(r, \mathbb{H})$. 
Let 
$$
J_{r} := 
\begin{pmatrix}
0 & -I_{r} \\ 
I_{r} & 0 
\end{pmatrix}. 
$$
Then the following are known (see, e.g., \cite{Aslaksen1996}): 
\begin{enumerate}
\item[(S$0'$)] the maps $\phi_{r}$ and $\psi_{r}$ are injective algebra homomorphisms; 
\item[(S$1'$)] $\Sdet{(a a')} = \Sdet{(a)} \: \Sdet{(a')}$; 
\item[(S$2'$)] $a$ is invertible in $\Mat(r, \mathbb{H})$ if and only if $\Sdet(a) \in \mathbb{C}$ is invertible; 
\item[(S$3'$)] if $a'$ is obtained from $a$ by adding a left-multiple of a row to another row or a right-multiple of a column to another column, 
then we have $\Sdet(a') = \Sdet(a)$; 
\item[(S$4'$)] $\Sdet(a) \in Z(\mathbb{H}) = \mathbb{R}$; 
\item[(S$5'$)] $\left( \Det_{\Mat(4r, \mathbb{R})}^{\mathbb{R}} \circ \phi_{2r} \circ \psi_{r} \right)(a) = \Sdet{(a)}^{2}$; 
\item[(S$6'$)] $\phi_{r}(\Mat(r, \mathbb{C})) = \{ c \in \Mat(2r, \mathbb{R}) \mid J_{r} c = c J_{r} \}$; 
\item[(S$7'$)] $\psi_{r}(\Mat(r, \mathbb{H})) = \{ b \in \Mat(2r, \mathbb{C}) \mid J_{r} b = \overline{b} J_{r} \}$. 
\end{enumerate}

These properties can be derived from results given in Sections~$2$--$8$. 
Let $A = \mathbb{H}$, 
let $B = \mathbb{C}$, 
let $C = \mathbb{R}$, 
let $e = (1 \: j)$, 
and let $f = (1 \: i)$. 
Then, the basis $e$ of $A$ as $B$-module and the basis $f$ of $B$ as $C$-module satisfy the conditions~(i)--(vi). 
From Examples~$\ref{rei:3.1}$ and $\ref{rei:3.2}$, 
we have $\phi_{r} = \iota \circ L_{f \otimes I_{r}}$ and $\psi_{r} = \iota' \circ L_{e \otimes I_{r}}$, 
where $\iota$ and $\iota'$ are inclusion maps. 
Therefore, (S$0'$) holds. 
Also, from Lemma~$\ref{lem:6.2}$, (S$1'$) and (S$2'$) hold. 
By Lemma~$\ref{lem:6.3}$, we have (S$3'$). 
From Corollary~$\ref{cor:7.3}$, (S$4'$) holds. 
By Corollary~$\ref{cor:7.4}$, (S$5'$) holds. 
Finally, (S$6'$) and (S$7'$) can be derived from Corollary~$\ref{cor:7.8}$ and the fact that $(j J_{r}) b = b (j J_{r})$ if and only if $J_{r} b = \overline{b} J_{r}$.

\section{On the relationship to the group determinant}
In this section, 
we recall the definition of group determinants, 
and we give an extension of Dedekind's theorem and derive an inequality for the degrees of irreducible representations of finite groups.

First, we recall the definition of the group determinant. 
Let $G$ be a finite group, 
let $[x_{g}] = \{ x_{g} \mid g \in G \}$ be the set of independent commuting variables, 
let $\mathbb{C}[x_{g}]$ be the polynomial ring in the variables $[x_{g}]$ with coefficients in $\mathbb{C}$, 
and let $|G|$ be the order of $G$. 
The group determinant $\Theta(G)$ of $G$ is given by $\Theta(G) := \Det_{\Mat(|G|, \mathbb{C}[x_{g}])}^{\mathbb{C}[x_{g}]}{\left( x_{g h^{-1}} \right)_{g, h \in G}} \in \mathbb{C}[x_{g}]$, 
where we apply a numbering to the elements of $G$ (for details, see, e.g., \cite{conrad1998origin}, \cite{Frobenius1968gruppen}, \cite{Frobenius1968gruppencharaktere}, \cite{Frobenius1968gruppen2}, \cite{Hawkins1971}, \cite{Johnson1991}, \cite{van2013history}, and \cite{Yamaguchi2017}). 
It is thus seen that the group determinant $\Theta(G)$ is a homogeneous polynomial of degree $\left| G \right|$. 
In general, the matrix $\left( x_{g h^{-1}} \right)_{g, h \in G}$ is covariant under a change in the numbering of the elements of $G$. 
However, the group determinant, $\Theta(G)$, is invariant. 

Let $\mathbb{C}G := \left\{ \sum_{g \in G} c_{g} g \mid c_{g} \in \mathbb{C} \right\}$ be the group algebra of $G$ over $\mathbb{C}$, 
let $H$ be an abelian subgroup of $G$, 
let $[G:H]$ be the index of $H$ in $G$, 
let $A = \mathbb{C}[x_{g}] \otimes \mathbb{C}G = \left\{ \sum_{g \in G} c_{g}g \mid c_{g} \in \mathbb{C}[x_{g}] \right\}$, 
let $B = \mathbb{C}[x_{g}] \otimes \mathbb{C}H$, 
and let $C = \mathbb{C}[x_{g}] \otimes \mathbb{C}\{ 1_{G} \}$. 
For a group $G$, we denote the regular representation of the group $G$ as $L_{G}$. 
We regard $L_{G}$ as $\mathbb{C}[x_{g}]$-algebra homomorphism from $A$ to $C$. 
Then, from Lemma~$\ref{lem:7.5}$, $L_{G}$ is equivalent to the left regular representation from $A$ to $\Mat(|G|, C)$. 
Therefore, the following commutative diagram holds: 
\[
\xymatrix{
A \ar[r]^-{L} \ar[d]^-{L_{G}} & \Mat([G:H], B) \ar[d]^-{L'} \ar[r]^{\Det_{\Mat([G:H], B)}^{B}} & B \ar[r]^-{L_{H}} & \Mat(|H|, C) \ar[d]_-{\Det_{\Mat(|H|, C)}^{C}} \\ 
\Mat(|G|, C) \ar@{^{(}-_>}[r]^-{} & \Mat(|H|, \Mat([G:H], C)) \ar@{^{(}-_>}[r]^-{} & \Mat(|G|, C) \ar[r]_-{\Det_{\Mat(|G|, C)}^{C}} & C \\ 
}
\]
It is easy to show that $\Theta(G) = \Det_{\Mat(|G|, \mathbb{C}[x_{g}])}^{\mathbb{C}[x_{g}]} \left( \sum_{g \in G} x_{g} L_{G}(g) \right)$ (see, e.g., \cite{conrad1998origin}). 
Therefore, we have 
\begin{align*}
\Theta(G) 1_{G} 
&= \left( \Det_{\Mat(|G|, C)}^{C} \circ L_{G} \right) (\mathfrak{X}_{G}) \\ 
&= \Sdet_{\Mat(1, A)}^{C}(\mathfrak{X}) \\ 
&= \left( \Sdet_{\Mat(1, B)}^{C} \circ \Sdet_{\Mat(1, A)}^{B} \right) (\mathfrak{X}_{G}) \\ 
&= \left( \Det_{\Mat(|H|, C)}^{C} \circ L_{H} \circ \Det_{\Mat([G:H], B)}^{B} \circ L \right) (\mathfrak{X}_{G}). 
\end{align*}

We extend $\varphi \in \widehat{G}$ to $\varphi \colon \mathbb{C}[x_{g}]G \to \mathbb{C}[x_{g}]G$ satisfy $\varphi\left(\sum_{g \in G} c_{g} g \right) = \sum_{g \in G} c_{g} \varphi(g)$, 
where $c_{g} \in \mathbb{C}[x_{g}]$. 
Frobenius proved the following theorem concerning the factorization of the group determinant (see, e.g., \cite{conrad1998origin}). 

\begin{thm}[Frobenius' theorem, Theorem~\ref{thm:1.6}]\label{thm:9.1}
Let $G$ be a finite group. 
Then we have the irreducible factorization 
$$
\Theta(G) = \prod_{\varphi \in \widehat{G}} \Det_{\Mat(\deg{\varphi}, \mathbb{C}[x_{g}])}^{\mathbb{C}[x_{g}]} \left( \varphi\left( \mathfrak{X}_{G} \right) \right)^{\deg{\varphi}}. 
$$
\end{thm}

Let $\widehat{G} = \{ \varphi^{(1)}, \varphi^{(2)}, \ldots, \varphi^{(s)} \}$ be a complete set of inequivalent irreducible representations of $G$ over $\mathbb{C}$. 
Theorem~$\ref{thm:9.1}$ holds from the following theorem (which is treated in detail in \cite{benjamin2013}). 

\begin{thm}[]\label{thm:9.2}
Let $G$ be a finite group, 
let $d_{i} = \deg{\varphi^{(i)}}$, 
and let $L_{G}$ be the left regular representation of $G$. 
Then we have 
$$
L_{G} \sim d_{1} \varphi^{(1)} \oplus d_{2} \varphi^{(2)} \oplus \cdots \oplus d_{s} \varphi^{(s)}. 
$$
\end{thm}

Therefore, the following theorem is deduced from Corollary~$\ref{cor:5.1}$. 

\begin{thm}[Extension of Dedekinds' theorem, Theorem~$\ref{thm:1.4}$]\label{thm:9.3}
Let $G$ be a finite group and let $H$ be an abelian subgroup of $G$. 
Then, writing $\Sdet_{\Mat(1, A)}^{B}(\mathfrak{X})$ as $\Theta(G:H)$, we have 
\begin{align*}
\Theta(G) 1_{G} 
&= \prod_{\varphi \in \widehat{G}} \Det_{\Mat(\deg{\varphi}, B)}^{B} \left( \varphi(\mathfrak{X}) \right)^{\deg{\varphi}} \\ 
&= \prod_{\chi \in \widehat{H}} \chi \left( \Theta(G:H) \right). 
\end{align*}
\end{thm}

The following is a special case of Theorem~$\ref{thm:9.3}$ \cite{Yamaguchi2017}. 

\begin{thm}[Theorem~$\ref{thm:1.7}$]\label{thm:9.4}
Let $G$ be a finite abelian group and let $H$ be a subgroup of $G$. 
Then, for every $h \in H$, there exists a homogeneous polynomial $c_{h} \in \mathbb{C}[x_{g}]$ such that $\deg{c_{h}} = [G:H]$ and 
$$
\Theta(G) 1_{G} = \prod_{\chi \in \widehat{H}} \sum_{h \in H} \chi(h) c_{h} h. 
$$
If $H = G$, then we can take $c_{h} = x_{h}$ for each $h \in H$. 
\end{thm}

From Theorems~$\ref{thm:5.3}$, $\ref{thm:9.1}$ and $\ref{thm:9.3}$, we obtain the following corollary: 

\begin{cor}[Corollary~$\ref{cor:1.8}$]\label{cor:10.5}
Let $G$ be a finite group and let $H$ be an abelian subgroup of $G$. 
Then, for all $\varphi \in \widehat{G}$, we have 
\begin{align*}
\deg{\varphi} \leq [G:H]. 
\end{align*}
\end{cor}

Note that Corollary~$\ref{cor:10.5}$ follows from Frobenius reciprocity, 
and it is known that if $H$ is an abelian normal subgroup of $G$, then $\deg{\varphi}$ divides $[G:H]$ for all $\varphi \in \widehat{G}$ (see, e.g, \cite{kondo2011group}).

\section{Future work}
There are several noncommutative determinants, and some of their relationships are known.
For example, 
the following is known (see, e.g., \cite{Aslaksen1996}, \cite{Kyrchei2008}, and \cite{Kyrchei-inbook}):
 
For any $M \in \Mat(r, \mathbb{H})$, we have 
$$
\Sdet{\left( M \right)} = \Mdet{\left(M M^{*} \right)} = \left( \Ddet{(M)} \right)^{2}, 
$$
where $\Mdet$ and $\Ddet$ are the determinant of the Moore (see, e.g., \cite{Aslaksen1996}, \cite{dyson1972quaternion}, and \cite{Kyrchei-inbook}) and of the Dieudonn\'{e} (\cite{artin2016geometric} and \cite{zbMATH03044809}), respectively. 

In the future, we study relations with other noncommutative determinants, 
for example, 
Capelli identities (see, e.g., \cite{Capelli1890}, \cite{itoh2001central}, \cite{Umeda2008}, and \cite{doi:10.1080/03081087.2017.1382443}), 
the determinant of Chen~\cite{Longxuan1991}, 
quasideterminant~\cite{Gel'fand1991}, 
the determinant of Kyrchei~\cite{Kyrchei2008} and \cite{Kyrchei-inbook}, 
the determinant of matrices of pseudo-differential operators~\cite{sato1975}, etc.

\clearpage

\thanks{Acknowledgments:}
I am deeply grateful to Professor Hiroyuki Ochiai, who provided the helpful comments and suggestions. 
Also, I would like to thank my colleagues in the Graduate School of Mathematics at Kyushu University, 
in particular Yuka Yamaguchi, for comments and suggestions. 
This work was supported by a grant from the Japan Society for the Promotion of Science (JSPS KAKENHI Grant Number 15J06842).

\bibliographystyle{plain}
\bibliography{reference}

\medskip
\begin{flushleft}
Naoya Yamaguchi\\
Office for Establishment of an Information-related School\\
Nagasaki University\\
1-14 Bunkyo, Nagasaki City 852-8521 \\
Japan\\
yamaguchi@nagasaki-u.ac.jp
\end{flushleft}

\end{document}